\documentclass[reqno,10pt]{amsart}
\usepackage[english]{babel}
\usepackage{amsthm}
\usepackage{amsmath}
\usepackage{amssymb}
\usepackage{amsfonts}
\usepackage{mathrsfs}
\usepackage{todonotes}
\usepackage{comment}

\usepackage[utf8]{inputenc}
\usepackage{latexsym}
\usepackage{verbatim}

\usepackage[all]{xy}
\usepackage{dsfont}
\usepackage[bookmarksnumbered,colorlinks,urlcolor=blue]{hyperref}
\usepackage{graphicx}
\usepackage{graphics}
\usepackage{float}
\usepackage{enumerate}
\def\bb#1\eb{\textcolor{blue}
{#1}} %
\def\br#1\er{\textcolor{red}
{#1}} %
\newcommand{\R}{\mathds R}
\newcommand{\N}{\mathds N}

\newcommand{\s}{\mathbb{S}}

   \def\br#1\er{\textcolor{red}{#1}} %
      \def\bb#1\eb{\textcolor{blue}{#1}} %

\title[Causal homotopy classes]{A note on the causal homotopy classes of a globally hyperbolic spacetime}

\author[P. Morales]{Pablo Morales Álvarez}
\author[M. S\'anchez]{Miguel S\'anchez}
\address{Departamento de Geometr\'{\i}a y Topolog\'{\i}a,  Universidad de Granada 
\hfill\break\indent
Facultad de Ciencias, Campus Fuentenueva s/n,
 \hfill\break\indent E-18071 Granada, Spain}
\email{pablomorales@correo.ugr.es sanchezm@ugr.es}

\begin{document}
\newtheorem{thm}{Theorem}[section]
\newtheorem{prop}[thm]{Proposition}
\newtheorem{lemma}[thm]{Lemma}
\newtheorem{cor}[thm]{Corollary}
\theoremstyle{definition}
\newtheorem{defi}[thm]{Definition}
\newtheorem{notation}[thm]{Notation}
\newtheorem{exe}[thm]{Example}
\newtheorem{conj}[thm]{Conjecture}
\newtheorem{prob}[thm]{Problem}
\newtheorem{rem}[thm]{Remark}
\newtheorem{conv}[thm]{Convention}
\newtheorem{crit}[thm]{Criterion}
\newtheorem{claim}[thm]{Claim}

\newcommand{\ben}{\begin{enumerate}}
\newcommand{\een}{\end{enumerate}}

\newcommand{\bit}{\begin{itemize}}
\newcommand{\eit}{\end{itemize}}

\begin{abstract}
Globally hyperbolic spacetimes admitting infinitely many causal (and timelike) homotopy classes of curves joining two prescribed points, are exhibited and discussed.
\end{abstract}

\maketitle 

\vspace*{-.5cm}

\section{Introduction}\label{section1}
In Lorentzian Geometry, the spaces $\mathcal{C}(p,q)$ and $C(p,q)$ containing, resp., all the causal and timelike curves connecting two causally related points $p<q$ become of crucial interest. In particular, it is well-known that globally hyperbolic spacetimes can be characterized as those causal spacetimes where all the spaces $\mathcal{C}(p,q)$ become compact for a suitable topology (see the review \cite{S_berlin});
moreover, in these spacetimes, each space $\mathcal{C}(p,q)$ always contains an element (necessarily a  causal pregeodesic) with maximum length. It is also natural to consider topological properties of $\mathcal{C}(p,q)$ and $C(p,q)$, such as their classes of homotopy and, more properly in the Lorentzian context, the classes of {\em causal} or {\em timelike} homotopies, where the homotopy is mediated via paths (the \emph{longitudinal curves} of the homotopy) that also lie in   
$\mathcal{C}(p,q)$ or $C(p,q)$. 

The purpose of the present note is  
{\em to exhibit examples of globally  hyperbolic spacetimes where $\mathcal{C}(p,q)$ and $C(p,q)$ contain  infinitely many classes of causal and timelike homotopy, resp.} The examples are robust, i.e., independent of details such as  dimensions, the global topology of the manifold or the existence of symmetries (the latter appear just to make easier  computations). Moreover, they do not rely on any ``mathematical trick'' about the smoothability of the considered curves. Indeed, the regularity of the curves will be optimal, in the sense that the  infinitely many non-homotopic curves to be obtained will be $C^\infty$ differentiable, 
but each two of these curves will not be connected by a causal homotopy even if the longitudinal curves of the possible homotopies are allowed to be just causal-continuous (the minimum regularity allowed for causal curves, see the section of preliminaries).

Even though quite a few of properties of these causal homotopy classes are known (see for example \cite{Beem, Gal, Law, MinSan}), we are not aware of such a type of examples. As we will see, 
they allow for an interesting comparison with the finiteness of  the (purely topological) homotopy classes for causal curves joining two points (which was established in \cite{Kim}), and they also allow for a better understanding of some subtleties related to the existence of either an {\em arbitrary close causal curve} or a {\em variation by causal curves}.

The paper is organized as follows. After some  precise definitions on homotopy classes and a technical preliminary lemma concerning Brouwer's topological degree in Section \ref{s2}, the examples  are constructed in Section \ref{s3}. The construction of the examples is rather intuitive and easy to understand; nevertheless, the formal proof of the intended properties is not so trivial, as it requires implicitly the previous topological tools included in section \ref{s2}.
The example for causal homotopy classes is described first, while the case of timelike homotopy classes will be a (somewhat burdensome) modification of the previous one. Finally, in section \ref{TopologicalAndDiscuss} we compare these examples with the result in \cite{Kim} about (purely topological) homotopies, and in section \ref{LastSect} a discussion on related questions and physical interpretations is carried out.

\section{Preliminaries} \label{s2}

We will follow usual conventions as in \cite{Beem, MinSan, ON}. In particular, $(M,g)$ denotes a spacetime, i.e. a time-oriented connected Lorentzian manifold. For points $p,q\in M$,  $p\neq q$, which are  causally related ($p< q$), $\mathcal{C}(p,q)$ will denote the set of all the  future-directed causal-continuous (called non-spacelike in \cite{Beem}) curves $\Gamma: [a,b]\rightarrow M$ from $p$ to $q$. Let us recall this concept, which is a natural generalization of the typical notions of {\em causal} for smooth curves, and can be also found in classical references such as \cite[p. 54]{Beem}.
\begin{defi}
Let $(M,g)$ be a spacetime and $I\subseteq \R$ an interval. A continuous curve $\gamma:I\to M$ is called \emph{future-directed causal-continuous} if for any $t\in I$ there exists a convex neighborhood\footnote{Recall that a convex neighborhood is a normal neighborhood of all of its points, cf. \cite[Section 5.2]{ON} or \cite[Section 3.1]{Beem}; the symbol $<_U$ denotes the causal relation regarding $U$ as a spacetime.} $U$ of $\gamma(t)$ and $\varepsilon>0$ such that $\gamma(t-\varepsilon,t+\varepsilon)\subset U$ and $$\gamma(t_1)<_U \gamma(t_2)\quad \forall \ t_1,t_2\in (t-\varepsilon,t+\varepsilon)\textrm{ with $t_1<t_2$}.$$  
\end{defi}

Past-directed causal-continuous curves are defined analogously and a causal-continuous curve is either a future-directed or a past-directed. It is well-known that these curves, even if non-smooth, can be assumed to be parametrized so that they are locally Lipschitz (and globally Lipschitz for any auxiliary Riemannian metric, as $[a,b]$ is compact) and, thus, absolutely continuous and  almost everywhere differentiable  (cf. \cite[p. 75]{Beem}, \cite[Appendix A]{CFS}). By convenience, we consider a prescribed interval $[a,b]$ (which could be taken equal to $[0,1]$). Then, $C(p,q)$ will denote the subset of  $\mathcal{C}(p,q)$ containing its piecewise-smooth   $C^1$ timelike curves. 
 
Such orders of regularity will be the most appropriate for the robustness of our examples, since the assumed low regularity for causal curves (which permits homotopies with low regularity too), will not avoid the existence of infinitely many causal homotopy classes, each one containing smooth curves. Moreover, the smooth timelike curves in infinitely many different timelike homotopy classes to be constructed will also remain in different causal homotopy classes (recall that, in general, a single causal homotopy class may have more than one timelike homotopy class\footnote{See \cite{MinSan} for this and other subtleties. Notice that causal chains (as if the curves were piecewise smooth) are used there, following the approach  detailed in \cite{Beem}.}).

Specifically, a {\em causal homotopy} between two curves $\Gamma_0, \Gamma_1 \in \mathcal{C}(p,q)$  is a continuous  mapping $H:[0,1]\times [a,b] \to M$ such that
\begin{enumerate}
\item $H(0,t)=\Gamma_0(t)$, $H(1,t)=\Gamma_1(t)$,
\item $t\mapsto H(s,t)=: \Gamma_s(t)$ belongs to $\mathcal{C}(p,q)$ for every $s\in [0,1]$.
\end{enumerate} 
In the case that such a $H$ exists, $\Gamma_0, \Gamma_1$ are said to be causally homotopic. 
Analogously, when $\Gamma_0, \Gamma_1 \in C(p,q)$ one can speak of a {\em timelike homotopy} by imposing additionally that $H$ is piecewise smooth and each $\Gamma_s\in C(p,q)$. Nevertheless, as explained above, we will not use this concept   but the one of (non-)causally homotopic curves. Of course, $\Gamma_0$ and $\Gamma_1$ are called just {\em homotopic} or, for emphasis here, {\em topologically homotopic}, when a continuous $H$ as above exists by relaxing  (2) into the weak condition that the endpoints of all $\Gamma_s$ are fixed (equal to $p,q$).

Now, we deal with a  preliminary topological question, which is summarized in Lemma \ref{Lemma}. In Remark \ref{r2.4} we explain how this result will be subsequently used as an essential part to check the properties of our examples. First, recall the following intuitive consequence of Brouwer's degree theory (cf. 
 \cite[Lemma 5]{FloSan}, see also Figure~\ref{SquareFig} for a geometrical interpretation):

\begin{thm}
\label{BrouwerDegree}
Let $a<b$, $c<d$, and $\theta_0<\theta<\theta_1$ be real numbers. Let $F:[a,b]\times [c,d]\to \R$ be a continuous function with 
\begin{gather*}
F(s,c)=\theta_0, 
\qquad
F(s,d)=\theta_1 
\end{gather*}
for all $s\in [a,b]$. Then, there exists a connected subset $\Lambda\subseteq F^{-1}(\theta)$ that intersects $\{a\}\times [c,d]$ and $\{b\}\times [c,d]$.
\end{thm}

\begin{figure}[H]
\centering
\includegraphics[scale=0.9]{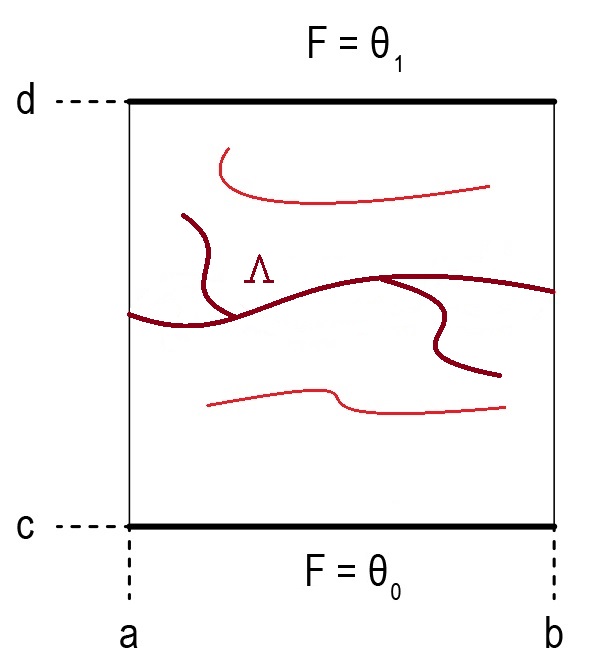}
\vspace{0.2cm}
\caption{Graphic representation for Theorem \ref{BrouwerDegree}.  For each $s\in [a,b]$ elementary Bolzano theorem ensures the existence of at least one $t\in(c,d)$ with $F(s,t)= \theta$. Theorem \ref{BrouwerDegree} goes further ensuring that such points $(s,t)$ (which appear in red) include a connected set $\Lambda$ (in bold red) which connects the vertical  lines $s=a$ and $s=b$.   When applied to Lemma \ref{Lemma}, $\theta_0$ and $\theta_1$ represent the polar angle of the endpoints of the longitudinal curves in the homotopy.}
\label{SquareFig}
\end{figure}
Let us apply this theorem in order to prove the (simple) property of the sphere stated in Lemma \ref{Lemma}. Let $\s^2 (\hookrightarrow \R^3)$ be the usual 2-sphere of radius 1, and use spherical coordinates $x=\sin\theta\cos\varphi$, $y=\sin\theta\sin\varphi$, $z=\cos\theta$, with polar
$\theta \in (0,\pi)$, and azimuth $\varphi \in (-\pi,\pi)$. The  meridian non-covered by these coordinates will be denoted $\varphi= \pm \pi$; however, the polar map $\theta: \s^2\rightarrow \R$ is well-defined and continuous everywhere. In the sequel, $Eq:=\theta^{-1}(\pi/2)$ will denote the equator, and $\mathcal{N}:=(0,0,1)$, $\mathcal{S}:=(0,0,-1)$ the north and south poles.

\begin{lemma}\label{Lemma}
Let $\alpha_0,\alpha_1:[0,\pi]\to \s^2$ be curves from $\mathcal{N}$ to $\mathcal{S}$ which parametrize, resp., the meridians of azimuthal angles $\varphi_0$, $\varphi_1\in(-\pi,\pi)$ for some $\varphi_0<\varphi_1$. Let $H:[0,1]\times [0,\pi] \to \s^2$ be a (topological) homotopy connecting $\alpha_0$ and $\alpha_1$. Then, the image of $H$ contains at least one of the two following sets: 
\begin{gather}
Eq \cap \left\{ p\in\s^2: \varphi(p)\in (\varphi_0,\varphi_1)\right\}, \label{eeq1} \\
Eq \cap \left(\left\{ p\in\s^2: \varphi(p) \in (-\pi,\varphi_0)\cup (\varphi_1,\pi)\right\} \cup \{\varphi=\pm \pi\}\right). \label{eeq2}
\end{gather}
\end{lemma}
\begin{proof}
Putting $F:=\theta\circ H:[0,1]\times [0,\pi]\to \R$, one has a continuous function with
$$F(s,0)=0\quad\textrm{and}\quad F(s,\pi)=\pi\quad\textrm{for every $s\in [0,1]$}.$$
Then, theorem \ref{BrouwerDegree} provides a connected subset $\Lambda$ of $F^{-1}(\pi/2)$ that intersects $\{0\}\times [0,\pi]$ and $\{1\}\times [0,\pi]$. Therefore, $H(\Lambda)$ is a connected subset of $Eq$  which contains points of azimuthal angle $\varphi_0$ and $\varphi_1$, and the result follows (lift and extend continuously the coordinate $\varphi$ to the natural covering $\R\rightarrow Eq \cong \s^1 $).
\end{proof}

\begin{rem}\label{r2.4}  Of course, the previous lemma can be improved because, as one would expect, even if the intersections with $Eq$ are not taken in \eqref{eeq1} and \eqref{eeq2}, one of the two regions delimited by meridians would be included in the image of $H$.  However, this would prolong  the proof\footnote{Say, if Theorem \ref{BrouwerDegree} was used for each parallel $\theta\in(0,\pi)$ as done for $Eq$, then one should justify that the same region $\varphi\in (\varphi_0,\varphi_1)$ or $\varphi\not\in (\varphi_0,\varphi_1)$ can be chosen with independence of $\theta$.} and will not be necessary. What we would like to stress is that, at least for the examples below, a kind of topological discussion becomes necessary\footnote{We have used Brouwer's but there are related topological theorems as Poincaré-Miranda (see the didactical overview \cite{KUL}) that would work too.}. Indeed, the application of  Lemma \ref{Lemma} in the next section will ensure (rigorously) that either the piece 
\eqref{eeq1} or the piece \eqref{eeq2} of the equator 
 must be swept when one meridian is deformed continuously   into another. As we will see, such a piece will be used to introduce an appropriate conformal factor on the sphere. The latter will force some longitudinal curves of any homotopy in $\s^2$ between the meridians to have a length bigger than $\pi$, and this property will imply the required  inexistence of causal homotopies between some lightlike geodesics in the spacetime $\R\times\s^2$.
\end{rem}

\section{Explicit examples} \label{s3}
Our examples will be provided in the globally hyperbolic spacetime $(M,g)$ given by a (naturally time-oriented) product Lorentzian manifold with $M=\R \times \s^2$, $g=-dt^2+ \pi^*g^\star$. Here, $t:M\rightarrow \R$, $\pi:M\rightarrow \s^2$ are the natural projections, 
and $g^\star=\Omega g_0$ is a metric on $\s^2$ conformal to its natural metric $g_0$ induced from $\R^3$, being $\Omega>0$ a conformal factor to be determined in each case.
\subsection{Causal homotopy classes}\label{subsecCausal}
Our aim in this point is to prove:

\begin{claim}
For some appropriate choice of the conformal factor $\Omega$, there exist infinitely many classes of causal homotopy between $(0,\mathcal{N})$ and $(\pi,\mathcal{S})$ in $(M,g)$.
\end{claim} 
\noindent In fact, the relevant properties to be satisfied by  $\Omega$ will be: 

\begin{enumerate}[(a)]
\item $\Omega\geq 1$.

\item $\Omega=1$ on the meridians of azimuthal angle $\varphi_n:=1/(n\pi)$, $n\in\N$. 

\item For some $0<\varepsilon_1<\varepsilon_2<\pi/2$ one has: (c1) $\Omega(\theta, \varphi)=1$ when $|\theta-\pi/2|>\varepsilon_2$ and (c2) $\Omega(\theta, \varphi)>1$ 
when $|\theta-\pi/2|<\varepsilon_1$ and $0<\varphi<\pi/2$, $\varphi\neq \varphi_n$, for all $n\in\N$.
\end{enumerate}

\noindent A concrete choice of such a factor 
$\Omega\in C^{\infty}(\s^2)$
would be
\begin{equation}\label{OriginalOmega}
\Omega (\theta, \varphi )= 1+f_1(\theta)f_2(\varphi), \qquad \qquad  \Omega(\theta, \pm\pi)=1,
\end{equation}
where $f_1\geq 0$ is any smooth function with compact support in $(0,\pi)$ and $f_1(\theta)=1$ in some neighborhood of $\theta= \pi/2$, and $f_2\geq 0$ is any smooth function with compact support in $(-\pi, \pi)$ satisfying\footnote{The $C^\infty$ smoothness of functions defined like $f_2$ around zero is a typical exercise of analysis, taking into account that, for all $m>0$, the function $e^{-1/x^2}/x^m$ and all its derivatives  vanish at $x=0$ (typically, this type of functions are used to construct bump functions).}
\begin{equation*}
f_2(\varphi) = e^{-1/\varphi^2}\sin^2(1/\varphi) \qquad \forall \varphi\in (-\pi/2,\pi/2), \quad (\varphi(0)=0).
\end{equation*}

In particular, the property (a) yields that, for any  curve\footnote{Due to the low regularity allowed for the longitudinal curves in causal homotopies, we will need to deal with absolutely continuous curves in $\s^2$. In this setting, the length of curves is still calculated as an integral, maintaining the same properties as for (piecewise) smooth curves  (cf. \cite{Anne} for a specific study of such properties, or Section 5.3 and Exercise 5.13 in \cite{Peter} for a general overview).} $\gamma$ on $\s^2$, the $g_0$-length $L_{g_0}(\gamma )$ and $g^\star$-length $L_{g^\star}(\gamma)$ satisfy   
\begin{equation}\label{RelationLengths}
L_{g_0}(\gamma)\leq L_{g^\star}(\gamma),  
\end{equation}
with strict inequality whenever $\gamma$ crosses a point with $\Omega>1$.

Let us now specify the infinite non-homotopic curves. For each $n\in \N$, let $\gamma_n:[0,\pi]\to \s^2$ be the $g_0$-arclength 
parametrized meridian of azimuthal angle $\varphi_n$ from $\mathcal{N}$ to $\mathcal{S}$. 
Lift each $\gamma_n$  to the  curve \begin{equation}
\Gamma_n(t)=(t,\gamma_n(t)), \qquad t\in [0,\pi],
\label{egamma}\end{equation} 
on the spacetime $M$. Notice that the curves $\Gamma_n$ are  lightlike (by the condition (b) below the claim) and they have the 
same endpoints $(0,\mathcal{N})$ and $(\pi,\mathcal{S})$. So, we have just to  show that all these curves lie in different causal homotopy classes.

\smallskip

\noindent {\em Proof of the causal homotopy non-equivalence}.
Assume that there existed indexes $i<j$  with $\Gamma_i$ and $\Gamma_j$ connected by a 
causal homotopy 
$\Upsilon_s(t)$, $s\in [0,1]$, $t\in [0,\pi]$. Projecting onto $\s^2$, one obtains a continuous mapping 
\begin{gather*}
H:[0,1]\times [0,\pi]\to \s^2 \\
(s,t)\mapsto \alpha_s(t),
\end{gather*}
where every $\alpha_s$ is a (absolutely continuous) curve from $\mathcal{N}$ to $\mathcal{S}$ and, as each $\Upsilon_s$ is causal, 
\begin{equation}\label{lengths}
L_{g^\star}(\alpha_s)\leq\pi \quad\forall s\in [0,1].
\end{equation}
Since according to \eqref{egamma} it is $\alpha_0=\gamma_i$ and $\alpha_1=\gamma_j$, we can apply lemma \ref{Lemma} to $H$, obtaining a  $s_0\in (0,1)$ such that $\alpha_{s_0}$ passes through a point with $\Omega > 1$ (recall property (c2) of $\Omega$). Then, inequality \eqref{RelationLengths} holds strictly, and together with \eqref{lengths} yields the contradiction 
\begin{equation*}
\pi=d_{g_0}(\mathcal{N},\mathcal{S})\leq L_{g_0}(\alpha_{s_0})<L_{g\star}(\alpha_{s_0})\leq \pi.
\end{equation*}

\subsection{Timelike homotopy classes} In the previous example we considered causal homotopy classes with no timelike curves. Next,  a new conformal factor $\tilde \Omega$ will be constructed as a small modification of previous $\Omega$ (see figure \ref{grafico1}),  so that those classes will contain timelike curves (which will remain in different causal homotopy classes).

Let  $\{\gamma_n\}_{n\in \N}$ be the same curves in $\s^2$ as  below \eqref{RelationLengths}, put 
 $q_n:=\gamma_n(\pi/2) (\in Eq)$, and let $p_n\in Eq$ be the middle points with $\varphi(p_n)=((1/n)+1/(n+1))/2\pi$ for all $n\in \N$.
Now, modify  $\Omega$ 
as $\tilde \Omega=\Omega-f$, where $f\in C^{\infty}(\s^2)$, $f\geq 0$ must vanish except on 
a small neighborhood $V_n$ of each $q_n$, being $f(q_n)>0$ and $f\leq \nu_n$ on each  $V_n$ for constants $\{\nu_n\} \searrow 0$. 
$V_n$ will be chosen as a $g_0$ ball, with small diameter $\varepsilon_n$ so that the closures $\overline{V}_n$ do not contain any $p_m$, they are pairwise disjoint, and  $\{\varepsilon_n\}\searrow 0$.

\begin{claim}\label{TechnicalLemma} The values $\nu_n, \varepsilon_n>0$ can be chosen small enough such that, calling again $g^\star=\tilde \Omega \cdot g_0$, no (absolutely continuous)  curve
$\gamma$ in $\s^2$ can simultaneously satisfy: 
\begin{enumerate}[(i)]
\item It connects $\mathcal{N}$ with $\mathcal{S}$.
\item\label{item2} It passes through $p_n$ for some $n\in \N$.
\item $L_{g^\star}(\gamma)\leq \pi$.
\end{enumerate}
\end{claim}
\noindent  To check this, recall that $\Omega(p_n)>1$ and take a neighborhood $U_n\ni p_n$  such that $\Omega\geq 1+\mu_n$ on $U_n$ for some constant $\mu_n>0$. Choose also $U_n$ with small $g_0$-diameter $\delta_n>0$ so that no $U_n$ intersect any $V_m$. Once constructed these sequences $\{\mu_n\}, \{\delta_n\} \rightarrow 0$, it suffices to take $\{\nu_n\}$ and $\{\varepsilon_n\}$ small enough in comparison with $\{\mu_n\}$ and $\{\delta_n\}$, in such a way that:

\begin{enumerate}

\item The  difference between the $g^\star$ and the $g_0$ lengths for going from $p_n$ to any point outside $U_n$ must be greater than the 
maximum difference between the $g_0$ and the $g^\star$ distances between any two points in $\overline{V}_{n\pm 1}$. More precisely, recall that, by   the Claim  \ref{TechnicalLemma} (\ref{item2}), the curve is forced to pass through some $p_n$; then, the $g^\star$-length of the portion of curve $\gamma$ in $U_n$ must exceed its $g_0$-length in a quantity bigger than minus the difference between the $g^\star$-length of any portion of the curve  $\gamma|_{[a',b']}$ 
included in $V_{n+1}$ or $V_{n-1}$ and the   $(g_0|_{V_{n \pm 1}})$-distance  between the endpoints of  $\gamma|_{[a',b']}$.


\item The $V_n$'s are far enough from each other (i.e. $\varepsilon_n$ are small enough) so that if $\gamma$ moves from one $V_i$ to another $V_j$ with, say, $i<j$, this is $g_0$ and  $g^\star$-longer than $2 \varepsilon_i$, i.e. the maximum $g_0$ length to  cross  $V_i$ by means of any meridian.   

\end{enumerate}

\noindent Notice that once these choices have been done, a straightforward discussion of cases for $\gamma$ (dependending on if it crosses none, one or more than one neighborhood $V_n$) shows that if it satisfied the three conditions in the claim then $\mathcal{N}$ and $\mathcal{S}$ could be connected with a curve of $g_0$-length less than $\pi$, a contradiction.

Once  Claim \ref{TechnicalLemma} has been established, the required example will be provided by 
the same curves $\gamma_n$ and $\Gamma_n$ introduced in \eqref{egamma} and above it. More precisely:

\begin{claim}\label{ExampleTimelike}
For any conformal factor $\tilde{\Omega}\in C^\infty(\s^2)$ fulfilling the conditions of Claim \ref{TechnicalLemma}, there exist infinitely many timelike curves from $(0,\mathcal{N})$ to $(\pi,\mathcal{S})$ in $(M,g)$ lying in different causal homotopy classes.
\end{claim}
\begin{proof}
Since $\tilde{\Omega}\leq \Omega$ and the inequality is strict at the points $q_n$, the curves $\Gamma_n$ are causal and their speed is timelike at the instant $t=\pi/2$ (thus, they are not lightlike pregeodesics). Then, each $\Gamma_n$ can be deformed to a timelike $\tilde{\Gamma}_n$ by means of a (fixed endpoints) 
variation of longitudinal timelike curves (see \cite[Proposition 10.46]{ON}). 
Let us check that 
all these $\{\tilde{\Gamma}_n\}_{n\in \N}$ lie in different causal (hence timelike) homotopy classes. 
Otherwise, there would be indexes $i<j$ with $\Gamma_i$ and $\Gamma_j$ in the same causal homotopy class, i.e. connected by a causal homotopy $\mathscr{H}:[0,1]\times [0,\pi]\to M$. Projecting $\mathscr{H}$ onto $\s^2$ and applying lemma \ref{Lemma}, we would find a curve $\gamma$ fulfilling the three incompatible properties in claim \ref{TechnicalLemma}. 
\end{proof}

\vspace{0.5cm}
\begin{figure}[H]
\centering

\includegraphics[width=0.4\textwidth]{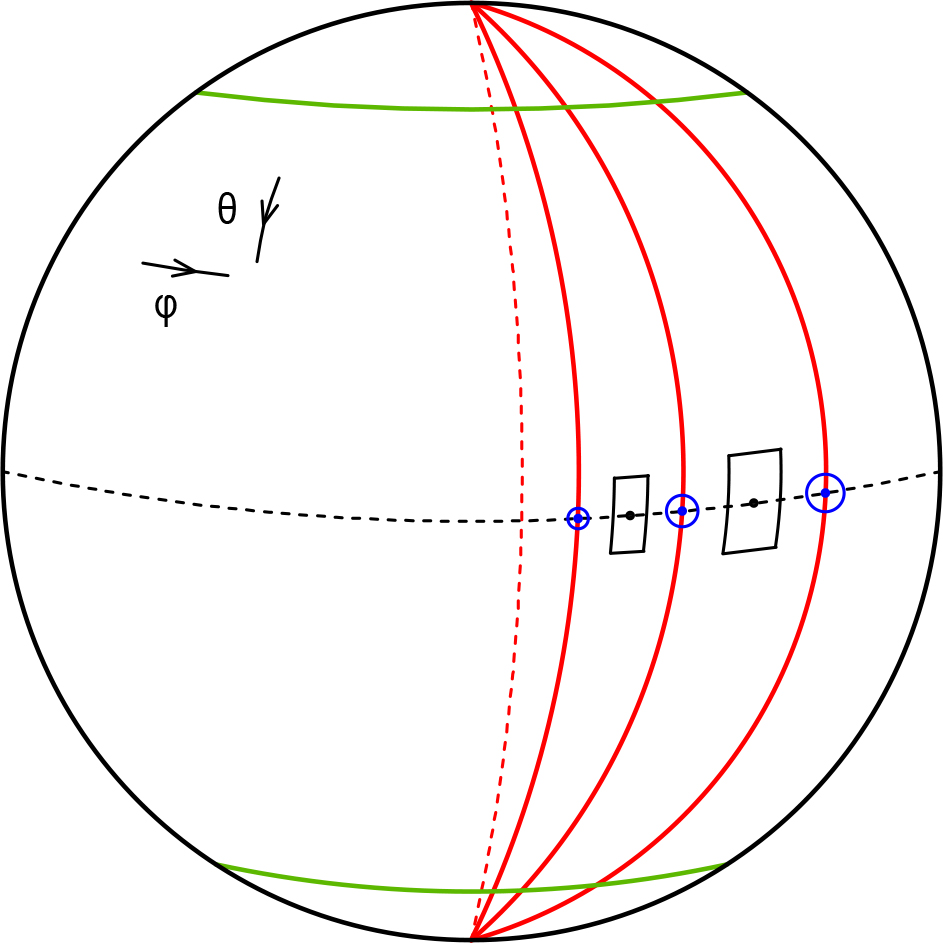}
\vspace{0.3cm}
\caption{Sketch for the conformal factors $\Omega, \tilde \Omega$. Red lines are the meridians with $\varphi=1/(n\pi)$. The green parallels $\theta=\pm \varepsilon_2 +\pi/2$ delimit polar caps where $\Omega (\equiv \tilde \Omega)\equiv 1$.  For $\tilde \Omega$ (timelike case), the points along the equator inside a black rectangle represent the points $p_n$'s inside the neighborhoods $U_n$'s, while the points inside a blue circle represent the $q_n$'s inside $V_n$'s.}
\label{grafico1}
\end{figure}


\section{Topological homotopy classes}\label{TopologicalAndDiscuss}

As we will comment also in the last remarks of the paper (see the third comment in section \ref{LastSect}), our examples become more relevant when comparing them with the properties of (purely topological) homotopy classes of causal curves. As we see in the following result by D.H. Kim \cite{Kim} (see also \cite{Kim2}), the behavior in this case is completely different.

\begin{thm}\label{FiniteNumber} \cite{Kim, Kim2}
Let $(M,g)$ be a globally hyperbolic spacetime, and $p,q\in M$ with $p< q$. Then, the number of (topological) homotopy classes in $\mathcal{C}(p,q)$ is finite.
\end{thm}
This can be also paraphrased by saying that the classes of homotopy 
of curves connecting $p$ and $q$ {\em which contain a causal curve} must be finite ---in contrast with the previous examples for causal or timelike classes. 
 Next,  an alternative proof
is provided, by developing a notion of tubular neighborhoods applicable to causal-continuous curves. 

Recall that, for any topological space $X$, a subset $C$ is called a {\em strong retract of deformation} if there exists a homotopy $H: [0,1]\times X \rightarrow X$ such that $H(0,\cdot)$ is the identity in $X$, $H(1,X)=C$ and $H(s,x)=x$ for all $x\in C$ and $s\in[0,1]$.   The 
next lemma is  geometrically intuitive, and makes cleaner our argument for the proof of  Theorem \ref{FiniteNumber}. We state it under the ambient hypothesis of global hyperbolicity for simplicity, but the proof  can be adapted easily to {\em embedded} causal-continuous curves
in any spacetime with some local work. 


\begin{lemma}\label{LemmaHomotopy}
Let $\Gamma: [a,b]\rightarrow M$ be a future-directed causal-continuous curve. Then, $\Gamma$ admits an extension (as a causal-continuous curve) $\tilde\Gamma: (\tilde a,\tilde b)\rightarrow M$, $[a,b]\subset (\tilde a,\tilde b)$, $\tilde \Gamma|_{[a,b]}=\Gamma$, and a neighborhood $V$ such that the image of $\tilde \Gamma$ 
is a strong deformation retract of $V$.
\end{lemma}

\begin{proof}
With no loss of generality (see \cite{BS}), assume that $M$ is a smooth product $\R\times S$,  with the natural projection $t:\R\times S\rightarrow \R$ a smooth time function which parametrizes $\Gamma$. Then, this is written as $\Gamma(t)=(t,\gamma(t))$, $t\in [a,b]$, and consider any extension $\tilde\Gamma(t)= (t,\tilde{\gamma}(t))$, $t\in \R$, as a continuous-causal curve (its domain to be restricted later); for example, concatenating $\Gamma$ at its endpoints with the integral curves of $\partial_t$. 

$V$ will be constructed as a sort of tubular neighborhood around $\tilde{\Gamma}$ (see Figure \ref{DeformationRetraction2}). 
 For every $s\in [a-1,b+1]$, choose a neighborhood $U_s$ 
of $\tilde\gamma(s)$ in $S$ which admits an homeomorphism 
$\psi_s$ with the usual open ball of radius $2$,  $B(0,2)\subset\R^{n-1}$. Then, for each $s$, take an interval $ I_s=(s-\varepsilon_s,s+\varepsilon_s)\subset (a-2,b+2)$  such that $\tilde\Gamma(I_s)$ is contained in the open ``tubular'' subset of $I_{s}\times U_{s}$ given as $C_s=C_s(I_s, 1)$, with
$$C_s(I_s, r):=\left\{\  \left(t,\psi_s^{-1}\left(  \psi_s(\tilde{\gamma}(t))+v  \right)\right) \; : \; t\in I_s, \; v\in B(0,r)  \right\}.$$
By compactness, 
take a finite subcover $\{C_{s_i},\; i=1,\dots, k\}$ of the image of $\tilde{\Gamma}|_{[a-1,b+1]}$
and, for each $i$, define a continuous map 
$$
H_i : M\times [i-1,i] \rightarrow M \qquad \qquad \hbox{such that:}
$$ 
\begin{enumerate}[(i)]
\item $H_i(\cdot,i-1)$ is the identity in $M$, and $H_i(\cdot,s)$ is the identity in $M\setminus C_{s_i}$ for any $s\in [i-1,i]$.
\item The restriction of $H_i$ to $C_{s_i}'\times [i-1,i]$ is  a strong deformation retraction of $C_{s_i}'$ into $\tilde \Gamma(t(C_{s_i}'))$, where $C_{s_i}'= C_{s_i}([s-\varepsilon_{s_i}+\delta_i,s+\varepsilon_{s_i}-\delta_i], r=1/3)$, 
with $\delta_i>0$ small  so that 
$[a-1,b+1]\subseteq \cup_{i=1}^k  (s-\varepsilon_{s_i}+\delta_i,s+\varepsilon_{s_i}-\delta_i)$.
\item $H_i(C_{s_i},i)\subseteq C_{s_i}$.
\end{enumerate}

Once these properties are fulfilled, put $V:= \cup_{i=1}^k C_{s_i}'$. The concatenation of the functions $H_i$ yields a new function $M\times [0,k] \rightarrow M$, whose restriction to  $V\times [0,k]$ is the desired homotopy. To check that  (i)--(iii) can be fulfilled, recall the natural deformation $D_i: C_{s_i}\times [0,1] \to C_{s_i}$

$$D_i((t,p),\lambda))= \bigg(t, \psi_{s_i}^{-1}\Big( \psi_{s_i}(\tilde{\gamma}(t))+ (1-\lambda)\big(\psi_{s_i}(p)-\psi_{s_i}(\tilde{\gamma}(t))\big)\Big)\bigg),
$$
and modify it with bump functions into some $\tilde D_i$ so that: (a) $\tilde D_i=D_i$ on $C'_{s_i}\times [0,1]$, (b) $\tilde D_i(\cdot , \lambda)$ is the identity outside  $C_{s_i}([s-\varepsilon_{s_i}+(\delta_i/2),s+\varepsilon_{s_i}-(\delta_i/2)], r=2/3)$ for all $\lambda\in[0,1]$. $\Box$

\begin{figure}[H]
\centering
\includegraphics[width=0.125\textwidth]{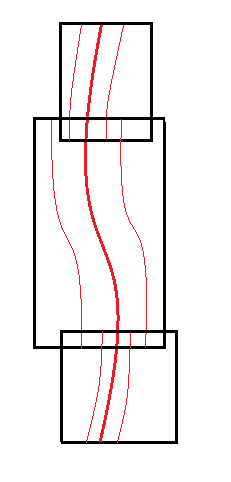}
\caption{Sketch for lemma \ref{LemmaHomotopy}. The red thick curve represents $\Gamma$, while thin red lines delimit the ``tubular'' subsets $C_{s}$. Black rectangles represent the products $I_{s}\times U_{s}$. }
\label{DeformationRetraction2}
\end{figure}

\vspace{-0.8cm}

\end{proof}

\noindent {\em Proof of Theorem \ref{FiniteNumber}.}
Assume that, for each $n\in \N$, $\Gamma_n:[0,1]\to M$ is a curve of $\mathcal{C}(p,q)$, with $\Gamma_i$ and $\Gamma_j$ non-homotopic whenever $i\neq j$. Global hyperbolicity guarantees the existence of a limit curve of the sequence (see \cite[Corollary 3.32]{Beem}), which is a future-directed causal-continuous curve $\Gamma:[0,1]\to M$ from $p$ to $q$. Moreover, there exists a subsequence $\{\Gamma_m\}_m\subseteq\{\Gamma_n\}_n$ such that $\{\Gamma_m\}$ converges to $\Gamma$ in the $C^0$-topology (see  \cite[Proposition 3.34]{Beem}). 
Thus, taking the neighborhood $V$ provided by  Lemma \ref{LemmaHomotopy}, all $\Gamma_m$ lies in $V$  for large $m$.
 Then, the retraction claimed for $V$ yields the contradiction  that all $\Gamma_i$ are homotopic to $\Gamma$ and, thus, to each other.
$\Box$

\section{Further remarks and conclusions}\label{LastSect}

Finally, let us discuss some mathematical aspects about the examples obtained in previous Section \ref{s3}, as well as provide some physical interpretations of them.

(1) First, notice the robustness of the examples provided, since the underlying arguments do not rely on the special geometry or the abundant symmetries of the 2-sphere ---this only simplifies the exposition. Clearly, the example also works in higher dimensions (replacing $\s^2$ by $\s^m$, $m>2$) or if the spacetime is not spatially compact (replace $\s^2$ by  $\s^2\times \R$, putting the natural metric on the $\R$ part). Indeed, an example can be constructed easily starting at any globally hyperbolic spacetime containing  two horismotically related points $p, q$ with focusing lightlike geodesics   (i.e., whenever $p$ and $q$ are connected by infinitely many lightlike geodesics  but not by any timelike curve) and, then, modifying conveniently the metric so that infinitely many of them cannot be pairwise connected by means of a causal homotopy.

(2) The sequence of causal (or timelike) curves $\Gamma_n$ in different causal homotopy classes has, indeed, a limit curve $\Gamma$, which also lies in a different class. Notice that, if infinitely many $\Gamma_n$'s remained in a single homotopy class, then the limit curve would remain in that same homotopy class.  This can be checked directly for causal homotopy classes (concatenate causal homotopies from each $\Gamma_n$ to $\Gamma_{n+1}$ and add $\Gamma$ as a final longitudinal curve)
as well as for timelike ones (by using timelike chains of points which can be thought as broken geodesics as in \cite{Beem}, \cite{MinSan}). This underlies the fact that there exist maximizers of the length and other functionals in each timelike or causal homotopy class  \cite{MinSan}. 
Moreover, most of the results in calculus  of variations for causal geodesics which concern the existence of an ``arbitrarily close causal curve'' of a prescribed one can be re-stated  as the existence of a causal homotopy with the original curve, even though this is not always explicitly said in the literature\footnote{Let us point out that, even though the treatment of variations by lightlike geodesics in O'Neill's book is excellent and can be regarded as the standard approach (see \cite[Chapter 10]{ON}), the existence of the homotopies is not always stated clearly there, and one has to go into the details of the proofs. In particular,  this includes \cite[Prop. 48]{ON} (a result that, as pointed out by  Galloway \cite{Gallo}, contains also a small gap of a different type in its proof, which can be solved  either by using \cite{Gallo} or
\cite[Th. 10.72]{Beem}, \cite[Prop. 4.5.12]{HE});
see \cite[Proposition 6.8]{AJ} for a detailed proof of the existence of causal homotopies in this case
(valid even in the more general setting of Lorentz-Finsler spacetimes).}.   
 
(3) There is a dramatic difference when one considers causal homotopies instead of topological ones as in \cite{Kim}. In the former case, the constraint on the longitudinal curves to be causal restrict the allowed homotopies and, so, permits the existence of more homotopy classes. Mathematically, this is a natural restriction in the Lorentzian setting, and opens the possibility of results beyond classical Riemannian analogies (compare with \cite{Mukuno}).  Physically, the constraint means that the longitudinal curves  must be {\em paths}, i.e. curves that, in principle, can represent (massive or massless) particles.  Being this constraint suitable for this setting, it is natural to wonder: if a path $\Gamma$ connecting two events $p, q$  is obtained as a limit  of another paths $\{\Gamma_n\}$, can the former be regarded as a deformation of the later? We have seen here that the answer is negative, as the curves $\Gamma_n$ can be regarded as arbitrarily close to $\Gamma$ but $\Gamma$ is not reachable from each $\Gamma_n$ by using paths. 

(4) This may affect the intuition of strong cosmic censorship hypothesis. Recall that, from a mathematical viewpoint, one says that a causal spacetime contains a {\em naked singularity} when it admits two points $p, q$ such that $J^+(p)\cap J^-(q)$ is not compact (so that a causal spacetime is said globally hyperbolic when it does not contain naked singularities). The physical interpretation of this property is as follows. If $J^+(p)\cap J^-(q)$ is not compact then it contains a nonconvergent sequence of points $\{r_n\}\subset J^+(p)\cap J^-(q)$ and, thus, any sequence of paths $\{\Gamma_n\}$ each one starting at $p$, crossing $r_n$ and ending at $q$ satisfy: there exist a limit path $\Gamma^+$ that starts at $p$ but it cannot arrive at $q$ (as well as a dual one $\Gamma^-$ ending at $q$ but which cannot start at $p$). This impossibility to arrive means that  $\Gamma^+$ ``disappears'' from the spacetime (i.e $\Gamma^+$ arrives to some sort of singularity) and, moreover, this disappearance can be observed from $q$ (the singularity is naked). One could expect that the path $\Gamma^+$ were reachable continuously from paths $\Gamma_n$ but this was known to be false. In fact, one could think that the failure of this property was also a consequence  of the existence of the singularity. However,  as we have seen, this is not the case: even when a spacetime is free of naked singularities, a path $\Gamma$ obtained as a limit curve  of paths $\Gamma_n$ may be non-reachable continuously from any $\Gamma_n$.

(5) A worthy property, known by Riemannian specialists, holds for the Riemannian part of our second example.
 Namely, in the first example,   the meridian of azimuthal angle $\varphi=0$ from $\mathcal{N}$ to $\mathcal{S}$ attains both, its first conjugate and cut point, at $\mathcal{S}$. Nevertheless, in the second one, its first conjugate point is still $\mathcal{S}$, but there exists a previous cut point, since meridians $\varphi=1/(n\pi)$ are shorter curves connecting $\mathcal{N}$ with $\mathcal{S}$. Moreover, the  cut point approaches the first conjugate point when one consider small neighboords of the half lune delimited by $\gamma$ and each $\gamma_n$. 

\section*{Acknowledgments}
The authors acknowledge Professor D.H. Kim, from Dankook Univ., his comments on reference \cite{Kim}.
They also acknowledge the support of the  project
MTM2013-47828-C2-1-P (Spanish MINECO with FEDER funds) and the first-named one also acknowledges the grant ``Beca de iniciación a la investigación para másteres 2015''  (U. Granada).

\end{document}